\documentclass[a4paper,10pt]{amsart}

\usepackage[english]{babel}
\usepackage[utf8]{inputenc}

\usepackage{amssymb}
\usepackage{amsmath}
\usepackage{amsthm}
\usepackage{bbm}

\usepackage{enumerate}
\usepackage{tikz}
\usepackage{marginnote}

\newcommand{\id}{I}

\newcommand{\up}{{\operatorname{u}}}

\theoremstyle{definition}
\newtheorem{definition}{Definition}[section]
\newtheorem{remark}[definition]{Remark}

\theoremstyle{plain}
\newtheorem{proposition}[definition]{Proposition}

\newtheorem{theorem}[definition]{Theorem}
\newtheorem{corollary}[definition]{Corollary}

\begin{document}

\title{A short note on band projections in partially ordered vector spaces}
\author{Jochen Gl\"uck}
%\address{...}
\email{jochen.glueck@alumni.uni-ulm.de}
\date{\today}

\begin{abstract}
	Consider an Archimedean partially ordered vector space $X$ with generating cone (or, more generally, a pre-Riesz space $X$). Let $P$ be a linear projection on $X$ such that both $P$ and its complementary projection $\id - P$ are positive; we prove that the range of $P$ is a band. This shows that the well-known concept of band projections on vector lattices can, to a certain extent, be transferred to the framework of ordered vector spaces.
\end{abstract}

\maketitle

\section{Introduction}

Consider a vector space $X$ over the real field and let $X_+$ be a \emph{cone} in $X$, by which we mean that $\alpha X_+ + \beta X_+ \subseteq X_+$ for all scalars $\alpha, \beta \in [0,\infty)$ and that $X_+ \cap (-X_+) = \{0\}$. The tuple $(X,X_+)$ is called a \emph{partially ordered vector space}, and the partial order $\le$ on $X$ referred to by this terminology is the order given by $x \le y$ if and only if $y-x \in X_+$.

Partially ordered vector spaces have been present in functional analysis since the first half of the 20th century, and a special focus has often been placed on an important special case called \emph{vector lattices} or \emph{Riesz spaces}, which are partially ordered vector spaces in which any two elements have an infimum. For an overview over the theory of vector lattices we refer to one of the classical monographs \cite{Luxemburg1971}, \cite{Zaanen1983}, \cite{Schaefer1974} and \cite{Meyer-Nieberg1991}.

The concepts of \emph{disjointness}, \emph{disjoint complements} and \emph{bands} are an important building block of the theory of vector lattices. In 2006  van Gaans and Kalauch \cite{Gaans2006} generalised those notions to partially ordered vectors spaces and, based on earlier work by van Haandel \cite{Haandel1993}, they could show that some important properties of bands remain true on partially ordered vector spaces under very mild assumptions on the space. This lead to a series of follow-up papers \cite{Gaans2008, Gaans2008a, Kalauch2015, KalauchPreprint1, KalauchPreprint3} where disjointness and bands in partially ordered vector spaces were studied in more detail; moreover, in \cite{Kalauch2018} it was demonstrated that this theory can be used to generalise results about disjointness preserving $C_0$-semigroups on Banach lattices to the setting of ordered Banach spaces.

In the present short note we consider a linear projection $P$ on an ordered vector space $(X,X_+)$ and assume that both $P$ and $\id - P$ are positive. We call such a projection an \emph{order projection} and, under the same mild assumptions on $(X,X_+)$ as in \cite{Gaans2006}, we show that the range of $P$ is a band (and in fact even a projection band) in $X$.
This opens the door for a theory of band projections on $(X,X_+)$ which is, to a certain extent, similar to the theory of band projections on vector lattices.

In Section~\ref{section:reminder} we recall some terminology and give a brief reminder of bands in partially ordered vector spaces. In Section~\ref{section:projections} we discuss order projections and prove our main result.

\section{Bands in partially ordered vector spaces} \label{section:reminder}

Let $(X,X_+)$ be a partially ordered vector space. This space is called \emph{directed} if, for all $x,y \in X$, there exists $z \in X$ such that $z \ge x$ and $z \ge y$. The cone $X_+$ is called \emph{generating} if $X_+ - X_+ = X$, and it is easy to see that $(X,X_+)$ is directed if and only if $X_+$ is generating. The partially ordered vector space $(X,X_+)$ is called \emph{Archimedean} if the inequality $nx \le y$ for two vectors $x,y \in X$ and all positive integers $n \in \mathbb{N} := \{1,2,\dots\}$ implies that $x \le 0$. If, for instance, $X$ carries a norm and $X_+$ is closed with respect to this norm, one can readily check that $(X,X_+)$ is Archimedean.

The concept of disjointness in $(X,X_+)$ can be defined as follows: two vectors $x,y \in X$ are called \emph{disjoint}, and we denote this by $x \perp y$, if
\begin{align*}
	\{x+y,-x-y\}^\up = \{x-y,y-x\}^\up;
\end{align*}
here, the set $S^\up := \{z \in X: \, z \ge x \text{ for all } x \in S\}$ denotes the set of all \emph{upper bounds} of any given set $S \subseteq X$. This definition of disjointness was introduced in \cite{Gaans2006} and is motivated by the fact that two elements $x$ and $y$ of a vector lattice are disjoint if and only if $\lvert x+y\rvert = \lvert x-y\rvert$. If ones used this fact, it is not difficult to see that the above definition of disjointness coincides with the classical concept of disjointness in case that $(X,X_+)$ is a vector lattice.

Some elementary properties of disjoint elements in a partially ordered vector space $(X,X_+)$ can be found in \cite[Section~1]{Gaans2006}. For instance, besides the evident fact that $x \perp y$ if and only if $y \perp x$, one also has $x \perp x$ if and only if $x=0$. Moreover, note that the zero vector is disjoint to every other vector in $X$. We will need the following characterisation of disjointness in case that both $x$ and $y$ are positive. Recall that if two elements $x,y \in X$ have a \emph{largest lower bound} in $X$ then this largest lower bound is uniquely determined and called the \emph{infimum} of $x$ and $y$.

\begin{proposition} \label{prop:disjointness-of-positive-vectors}
	Two vectors $0 \le x,y \in X$ are disjoint if and only if their infimum exists and equals $0$.
\end{proposition}
\begin{proof}
	We first note that $\{x+y,-x-y\}^\up = \{x+y\}^\up$ (since $-x-y \le 0 \le x+y$) and that $x+y$ is an upper bound of both $x-y$ and $y-x$, so $\{x+y\}^\up \subseteq \{x-y,y-x\}^\up$.
	
	Now assume that $x$ and $y$ have an infimum and that this infimum coincides with $0$. Let $f$ be an upper bound of $x-y$ as well as of $y-x$. Then we have $y \ge x-f$ as well as $x \ge y-f$. This implies that $2y \ge x+y-f$ and $2x \ge x+y-f$, so $(x+y-f)/2$ is a lower bound of both $x$ and $y$ and thus, this vector is negative. Hence, $x+y \le f$.
	
	Now assume instead that $\{x+y\}^\up = \{x-y,y-x\}^\up$. Clearly, $0$ is a lower bound of both $x$ and $y$ and we have to show that it is the largest lower bound. So let $f \in X$ be another lower bound of $x$ and $y$. Then $x-f$ and $y-f$ are both positive, so we conclude that
	\begin{align*}
		x+y-f \ge x \ge x-y \qquad \text{and} \qquad x+y-f \ge y \ge y-x.
	\end{align*}
	Hence, $x+y-f$ is an upper bound of $\{x-y,y-x\}$ and thus, by assumption, also an upper bound of $x+y$. Therefore, $f \le 0$.
\end{proof}

We will also need the following simple fact which can, for instance, be derived from Proposition~\ref{prop:disjointness-of-positive-vectors} above.

\begin{proposition} \label{prop:disjointness-smaller-vectors}
	Let $x_1,x_2 \in X$. If $0 \le x_1 \le x_2$ and $x_2$ is disjoint to a positive vector $y \in X$, then $x_1$ is disjoint to $y$, too.
\end{proposition}

Now we finally come to the concept of disjoint complements. For each subset $S \subseteq X$ the set
\begin{align*}
	S^\perp := \{x \in X: \, x \perp y \text{ for all } y \in S \}
\end{align*}
is called the \emph{disjoint complement} of $S$. Inspired by the theory of vector lattices one suspects this concept to be quite useful, but only if one can prove that $S^\perp$ is always a vector subspace of $X$. It was shown in \cite[Corollary~2.2 and Section~3]{Gaans2006} that this is true on a very large class of partially ordered vector spaces, namely on each so-called \emph{pre-Riesz space}. This concept is originally due to van Haandel \cite{Haandel1993}; for a precise definition of pre-Riesz spaces we refer for instance to \cite[Definition~1.1(viii)]{Haandel1993} or \cite[Definition~3.1]{Gaans2006}. Here, we only recall that every pre-Riesz space has generating cone and that, conversely, every partially ordered vector space with generating cone which is, in addition, Archimedean is a pre-Riesz space; see \cite[Theorem~3.3]{Gaans2006} or Haandel's original result in \cite[Theorem~1.7(ii)]{Haandel1993} (but note that Haandel uses a somewhat different terminology: he uses that notion \emph{integrally closed} for what we call \emph{Archimedean}, and he uses the term \emph{Archimedean} for another, weaker property).

Let $(X,X_+)$ be a pre-Riesz space; as mentioned above, this implies that the disjoint complement $S^\perp$ of any subset $S$ of $X$ is a vector subspace of $X$. A subset $B \subseteq X$ is called a \emph{band} if the set $B^{\perp\perp} := (B^\perp)^\perp$ equals $B$ (see \cite[Definition~5.4]{Gaans2006}). For every $S \subseteq X$ the disjoint complement $S^\perp$ is itself a band \cite[Proposition~5.5(ii)]{Gaans2006}. Moreover, every band $B \subseteq X$ is \emph{solid} in the sense of \cite[Definition~2.1]{Gaans2003} or \cite[Definition~5.1]{Gaans2006}; this was proved in \cite[Proposition~5.3]{Gaans2006}. We call a band $B \subseteq X$ \emph{directed} if and only if for all $x,y \in B$ there exists $z \in B$ such that $z \ge x$ and $z \ge y$; equivalently, the positive cone $B_+ := B \cap X_+$ in $B$ fulfils $B_+-B_+ = B$.

Inspired by the theory of vector lattices we call a band $B \subseteq X$ a \emph{projection band} if $X = B \oplus B^\perp$, i.e.\ if $X$ is the direct sum of $B$ and its disjoint complement $B^\perp$. The following observation follows immediately from the definition of the notions \emph{band} and \emph{projection band}.

\begin{proposition} \label{prop:disjoint-complement-of-a-projection-band}
	Let $(X,X_+)$ be a pre-Riesz-space and let $B \subseteq X$ be a band. Then $B$ is a projection band if and only if $B^\perp$ is a projection band.
\end{proposition}

If a projection band $B \subseteq X$ is given, we can define a linear projection $P$ onto $B$ along $B^\perp$. This gives rise to the following terminology: a linear projection $P: X \to X$ on a pre-Riesz space $(X,X_+)$ is called a \emph{band projection} if there exists a projection band $B \subseteq X$ such that $P$ is the projection onto $B$ along $B^\perp$; it follows from Proposition~\ref{prop:disjoint-complement-of-a-projection-band} that a linear projection $P$ is a band projection if and only if $\id - P$ is a band projection. If $B \subseteq X$ is a projection band, then there exists, of course, only one band projection with range $B$, and this projection is called \emph{the} band projection onto $B$.

Recall that a linear operator on a partially ordered vector space $(X,X_+)$ is called \emph{positive} if it maps the cone $X_+$ into itself. The following proposition shows, among other things, that band projections are always positive.

\begin{proposition} \label{prop:band-projections-are-positive}
	Let $(X,X_+)$ be a partially ordered vector space.
	\begin{enumerate}[\upshape (a)]
		\item Assume that a vector $0 \le x \in X$ can be written as $x = y+z$ for disjoint vectors $y,z \in X$. Then $x$ and $y$ are positive, too.
		\item If $(X,X_+)$ is a pre-Riesz space and $P: X \to X$ is a band projection, then $P$ and $\id - P$ are positive.
	\end{enumerate}
\end{proposition}
\begin{proof}
	(a) Since $x = y+z$ and $x \ge 0 \ge -x = -y-z$, we conclude that $x$ is an upper bound of $\{y+z,-y-z\}$ and thus, by the disjointness of $y$ and $z$, also an upper bound of $\{y-z,z-y\}$. Hence,
	\begin{align*}
		y+z = x \ge y-z \qquad \text{and} \qquad y+z = x \ge z-y,
	\end{align*}
	which implies $z \ge 0$ and $y \ge 0$.
	
	(b) This is a consequence of~(a).
\end{proof}

Let $(X,X_+)$ be a partially ordered vector space and let $P: X \to X$ be a linear projection. If $P$ is positive, then we have $PX_+ = PX \cap X_+$, so the cone $PX_+$ on the space $PX$ induces the same order on $PX$ as the order inherited from $X$. Moreover, if $X_+$ is generating in $X$, then $PX_+$ is generating in $PX$. This observation is employed in the proof of the following proposition and it is also used in Section~\ref{section:projections} below.

\begin{proposition} \label{prop:directed-projection-bands}
	Let $(X,X_+)$ be a pre-Riesz space. Then every projection band in $X$ is directed.
\end{proposition}
\begin{proof}
	If $B \subseteq X$ is a projection band, then the band projection onto $X$ is positive according to Proposition~\ref{prop:band-projections-are-positive}(b). Since every pre-Riesz space has generating cone, the assertion follows from the discussion right before Proposition~\ref{prop:directed-projection-bands}.
\end{proof}

\section{Order projections are band projections} \label{section:projections}

We call a linear projection $P$ on a partially ordered vector space $(X,X_+)$ an \emph{order projection} if both $P$ and $\id - P$ are positive. Thus, $P$ is an order projection if and only if $0 \le Px \le x$ for all $x \in X_+$. Proposition~\ref{prop:band-projections-are-positive}(b) above shows that every band projection on a pre-Riesz space is an order projection. In Theorem~\ref{thm:main} below we prove that the converse implication is also true.

But first, we note the following result which is true on every partially ordered vector space with generating cone, be it pre-Riesz or not.

\begin{proposition} \label{prop:at-most-one-order-projection}
	Let $(X,X_+)$ be a partially ordered vector space with generating cone. If two order projections $P$ and $Q$ have the same range, then $P = Q$.
\end{proposition}

Proposition~\ref{prop:at-most-one-order-projection} was proved for ordered Banach spaces in \cite[Proposition~2.1.4]{GlueckWolffLB}. The proof for partially ordered vector spaces is virtually the same; we include it here for the sake of completeness:

\begin{proof}[Proof of Proposition~\ref{prop:at-most-one-order-projection}]
	It suffices to prove that $P$ and $Q$ have the same kernel. To this end, first assume that $0 \le x \in \ker P$. We have $0 \le Qx \le x$; as $Qx$ is contained in the range of $P$ it follows that
	\begin{align*}
		0 \le Qx = PQx \le Px = 0,
	\end{align*}
	so indeed $Qx = 0$. Thus, $Q$ maps every vector in $X_+ \cap \ker P$ to $0$. Moreover, $\ker P$ is the range of the positive projection $\id - P$ and the cone $X_+$ is generating in $X$; hence every vector in $\ker P$ can be written is a difference of two positive vectors in $\ker P$ and thus, $Q$ maps the entire space $\ker P$ to $0$. This proves that $\ker P \subseteq \ker Q$. By interchanging the roles of $P$ and $Q$ we also obtain $\ker Q \subseteq \ker P$ which proves the assertion.
\end{proof}

The following theorem is the main result of this short note.

\begin{theorem} \label{thm:main}
	Let $(X,X_+)$ be a pre-Riesz space. Then a linear projection $P: X \to X$ is an order projection if and only if it is a band projection.
\end{theorem}
\begin{proof}
	If $P$ is a band projection, then it is also an order projection by Proposition~\ref{prop:band-projections-are-positive}(b). So assume conversely that $P$ is an order projection. It will be convenient to denote the complementary projection of $P$ by $Q := \id - P$. Then $0 \le P,Q \le \id$.
	
	We are going to prove that $PX = (QX)^\perp$, and to this end we first show that $PX \cap X_+ = (QX \cap X_+)^\perp \cap X_+$. In the following, we will tacitly use the characterisation of disjointness of positive vectors given in Proposition~\ref{prop:disjointness-of-positive-vectors}.
	
	``$\subseteq$'' Let $0 \le x \in PX$ and let $0 \le y \in QX$; we have to prove that $x \perp y$. Clearly, $0$ is a lower bound of both $x$ and $y$, so let $f \in X$ be another lower bound of those vectors. Then $Pf \le Py = 0$ and $Qf \le Qx = 0$, so $f = Pf + Qf \le 0$. Hence, $x$ and $y$ have infimum $0$ and are therefore disjoint. 
	
	``$\supseteq$'' Let $x \in X_+$ and assume that $x$ is disjoint to every vector $0 \le y \in QX$. We have to show $x \in PX$ which is equivalent to showing $Qx = 0$. The vector $Qx$ is positive and contained in $QX$, so it is disjoint to $x$. However, the vector $Qx$ is located between $0$ and $x$, so it follows from Proposition~\ref{prop:disjointness-smaller-vectors} that $Qx \perp Qx$ and thus, $Qx = 0$.
	
	Now we can show that actually $PX = (QX)^\perp$:
	
	``$\subseteq$'' Let $x \in PX$ and let $y \in QX$. We have to show that $x$ and $y$ are disjoint. The positive cone $(PX)_+ := PX \cap X_+ = PX_+$ in $PX$ is generating in $PX$ since $X_+$ is generating in $X$. Similarly, $(QX)_+ := QX \cap X_+$ is generating in $QX$. Hence, we can find positive vectors $x_1,x_2 \in (PX)_+$ and $y_1,y_2 \in (QX)_+$ such that $x = x_1-x_2$ and $y = y_1 - y_2$. The vector $x_1$ is disjoint to both $y_1$ and $y_2$, so it is also disjoint to $y$, and for the same reason the vector $x_2$ is disjoint to $y$. Thus, $x = x_1-x_2$ is disjoint to $y$.
	
	``$\supseteq$'' Let $x \in (QX)^\perp$. By the inclusion that we have proved right above, the vector $Px$ is also contained in $(QX)^\perp$ and since $(QX)^\perp$ is a vector subspace of $X$, we conclude that $Qx = x-Px$ is contained in $(QX)^\perp$, too. In particular, $Qx$ is disjoint to itself, so $Qx = 0$. Hence, $x = Px \in PX$.
	
	We have thus shown that $PX = (QX)^\perp$, and now it is easy to deduce the assertion of the theorem: by interchanging the roles of $P$ and $Q$ we also obtain $QX = (PX)^\perp$ and hence, $(PX)^{\perp\perp} = (QX)^\perp = PX$. Thus, $PX$ is a band. Moreover, we have $X = PX \oplus QX = PX \oplus (PX)^\perp$, so $PX$ is a projection band. Since $P$ is the projection onto $PX$ along $QX = (PX)^\perp$, we conclude that $P$ is a band projection.
\end{proof}

\begin{remark}
	In \cite[Section~2]{GlueckWolffLB} a linear projection $P$ on an ordered Banach space with generating cone is called a \emph{band projection} if both $P$ and $\id - P$ are positive. Theorem~\ref{thm:main} shows that this terminology is consistent with the terminology used in the present short note (note that an ordered Banach space in the sense of \cite{GlueckWolffLB} has closed cone and is thus Archimedean; hence, such a space is pre-Riesz in case that it has a generating cone).
\end{remark}

Let us close this note with the following consequence of Theorem~\ref{thm:main}.

\begin{corollary}
	Let $(X,X_+)$ be a pre-Riesz space. If $V \subseteq X$ is a vector subspace of $X$ and $X = V + V^\perp$, then $V$ is a projection band.
\end{corollary}
\begin{proof}
	First note that $V \cap V^\perp = \{0\}$, so the assumption implies that $X = V \oplus V^\perp$. By $P: X \to X$ we denote the projection onto $V$ along $V^\perp$. Let $0 \le x \in X$; since the vectors $Px$ and $(\id-P)x$ are disjoint, it follows from Proposition~\ref{prop:band-projections-are-positive}(a) that both $Px$ and $(\id-P)x$ are positive. Hence, $P$ is an order projection and thus a band projection according to Theorem~\ref{thm:main}. This readily implies that $V = PX$ is a projection band.
\end{proof}

\subsection*{Acknowledgement}

The idea for Theorem~\ref{thm:main} note arose during a very pleasant visit of the author at the Institute of Analysis at Technische Universit\"at Dresden. The author wishes to thank the members of the institute for their kind hospitality and for many stimulating and enlightening discussions.

\bibliographystyle{plain}
\bibliography{literature}

\end{document}